\documentclass[a4paper,11pt]{amsart}
\usepackage{mathrsfs,amsmath,amssymb,amsthm,mathabx}

\newcommand{\myforall}{~~{\rm for\ all}~~}

\newcommand{\myforsome}{~~{\rm for\ some}~~}

\newcommand{\myiff}{~~{\rm if\ and\ only\ if}~~}

\newcommand{\conjugacy}[2]{{#2} \circ {#1} \circ {#2}^{-1}}

\newcommand{\Per}{{\rm Per}}
\newcommand{\percon}[2]{\Per({#1}) \subseteq \Per({#2})} 

\newcommand{\tfullshiftt}{two-sided full shift\ }

\newcommand{\tsubshiftt}{two-sided subshift\ }
\newcommand{\tsubshift}{two-sided subshift}

\newcommand{\kinsetu}{\vartriangleright}
\newcommand{\Z}{\mathbf{Z}} 
\newcommand{\Posint}{\Z_+}
\newcommand{\beposint}{\in \Posint}
\newcommand{\beint}{\in \Z}

\newcommand{\diam}{{\rm diam}}
\newcommand{\mesh}{{\rm mesh}}

\newcommand{\SFT}{{\rm SFT}} 

\newcommand{\Hom}{\mathcal{H}} 
\newcommand{\End}{\mathcal{H^+}} 
\newcommand{\U}{\mathscr{U}} 
\newcommand{\V}{\mathscr{V}}
\newcommand{\W}{\mathscr{W}}

\newcommand{\kuu}{\emptyset}
\newcommand{\nekuu}{\neq \kuu} 

\newcommand{\barphi}{\bar{\phi}}

\newcommand{\enuma}{\begin{enumerate}}
\newcommand{\enumz}{\end{enumerate}}

\newtheorem{thm}{Theorem}[section]
\newtheorem{lem}[thm]{Lemma}

\newtheorem{cor}[thm]{Corollary}

\theoremstyle{definition}

\theoremstyle{remark}

\numberwithin{equation}{section}

\begin{document}

\title[APERIODIC HOMEOMORPHISMS APPROXIMATE]{APERIODIC HOMEOMORPHISMS APPROXIMATE\\
CHAIN MIXING ENDOMORPHISMS \\
ON THE CANTOR SET}

\author{TAKASHI SHIMOMURA}

\address{Nagoya Keizai University, Uchikubo 61-1, Inuyama 484-8504, Japan}
\curraddr{}
\email{tkshimo@nagoya-ku.ac.jp}
\thanks{}



\date{April 30, 2011}

\dedicatory{}

\commby{}

\begin{abstract}
Let $f$ be a chain mixing continuous onto mapping from the Cantor set onto itself.
Let $g$ be an aperiodic homeomorphism on the Cantor set. We show that homeomorphisms that are topologically conjugate to g approximate $f$ in the topology of uniform convergence if a trivial necessary condition on periodic points is satisfied.
In particular, let $f$ be a chain mixing continuous onto mapping from the Cantor set onto itself with a fixed point and g, an aperiodic homeomorphism on the Cantor set. Then, homeomorphisms that are topologically conjugate to $g$ approximate $f$.
\end{abstract}

\maketitle

\begin{center} PREPRINT
\end{center}

\section{Introduction}
Let $(X,d)$ be a compact metric space.
Let $\End(X)$ be the set of all continuous mappings from $X$ onto itself.
In this manuscript, the pair $(X,f)$ ($f \in \End(X)$) is called a {\it topological dynamical system}. We mainly consider the case in which $X$ is homeomorphic to the Cantor set, denoted by $C$.
For any $f,g \in \End(X)$, we define $d(f,g):=\sup_{x \in X }d(f(x),g(x))$.
Then, $(\End(X),d)$ is a metric space of uniform convergence.
Let $\Hom(X)$ be the set of all homeomorphisms from $X$ onto itself.
Let $X$ be homeomorphic to $C$.
$\SFT(X)$ denotes the set of all $f \in \Hom(X)$ that are topologically conjugate to some \tsubshiftt of finite type.
T.~Kimura \cite[Theorem 1]{Kimura} and I \cite{Shimomura} have shown that elements in $\Hom(C)$ are approximated by expansive homeomorphisms with the pseudo-orbit tracing property.
$\SFT(C)$ coincides with the set of all expansive $f \in \Hom(C)$ with the pseudo-orbit tracing property (P.~Walters \cite[Theorem 1]{Walters}).
Therefore, $\SFT(C)$ is dense in $\Hom(C)$.
Fix $f \in \Hom(C)$.
Homeomorphisms that are topologically conjugate to $f$ will approximate some other homeomorphisms.
Let $(X,f)$ be a topological dynamical system. $x \in X$ is called a periodic point of period $n$ if $f^n(x) = x$.
Let $\Per(X,f) := \{n \beposint ~|~ f^n(x) = x \myforsome x \in X\}$, where $\Posint$ denotes the set of all positive integers.
Let $(X,f)$ and $(Y,g)$ be topological dynamical systems.
In this manuscript, a continuous mapping $\phi : Y \to X$ is said to be {\it commuting} if $\phi \circ g = f \circ \phi$ holds.
We write $(Y,g) \kinsetu (X,f)$ if there exists a sequence of homeomorphisms $\{\psi_k\}_{k=1,2,\dots}$ from $Y$ onto $X$ such that
 $\conjugacy{g}{\psi_k} \to f$ as $k \to \infty$.
Suppose that $(Y,g) \kinsetu (X,f)$ and that $g^n$ has a fixed point for some positive integer $n$. Then, $f^n$ must also have a fixed point.
Therefore, we get $\percon{Y,g}{X,f}$.
Let $\delta >0$.
A sequence $\{x_i\}_{i = 0,1,\dots,l}$ of elements of $X$ is a {\it $\delta$ chain} from $x_0$ to $x_l$ if $d(f(x_i),x_{i+1}) < \delta$ for all $i = 0,1,\dots,l-1$.
Then, $l$ is called the length of the chain.
A topological dynamical system $(X,f)$ is {\it chain mixing} if for every $\delta >0$ and for every pair $x,y \in X$, there exists a positive integer $N$ such that for all $n \geq N$, there exists a $\delta$ chain from $x$ to $y$ of length $n$.
Let $(\Lambda,\sigma)$ be a \tsubshiftt such that $\Lambda$ is homeomorphic to $C$. Let $X$ be homeomorphic to $C$ and $f$, a chain mixing element of $\End(X)$.
In a previous paper \cite[Theorem 1.1]{Shimomura1}, it was shown that the following conditions are equivalent:

\enuma
\item $\percon{\Lambda,\sigma}{X,f}$;
\item $(\Lambda,\sigma) \kinsetu (X,f)$.
\enumz
Let $(Y,g)$ be a topological dynamical system and $n \beposint$.
In this manuscript, we say that $g$ is {\it periodic} of period $n$ if $g^n = id_Y$, where $id_Y$ denotes the identity mapping on $Y$.
We say that $g$ is {\it aperiodic} if $g$ is not periodic.
Suppose that $g \in \Hom(Y)$ is periodic of period $n$ and that $(Y,g) \kinsetu (X,f)$ for some $f \in \End(X)$.
Then, it is easy to check that $f$ is also periodic of period $n$.
Note that even if $g$ is aperiodic, all the orbits of $g$ may be periodic.
This may happen if $g$ has periodic points of least period $n$ for infinitely many $n \beposint$.
In this manuscript, we shall show the following:
\begin{thm}\label{thm:main}
Let $X$ and $Y$ be homeomorphic to $C$; $f \in \End(X)$, chain mixing; and $g \in \Hom(Y)$, aperiodic.
Then, the following conditions are equivalent:
\enuma
\item $\percon{Y,g}{X,f}$;
\item $(Y,g) \kinsetu (X,f)$.
\enumz
\end{thm}

In the previous theorem, suppose that $f$ has a fixed point.
Then, $\Per(X,f) = \Posint$.
Therefore, the following corollary is obtained:
\begin{cor}\label{cor:1}
Let $X$ and $Y$ be homeomorphic to $C$; $f \in \End(X)$, chain mixing; and $g \in \Hom(Y)$, aperiodic.
Suppose that $f$ has a fixed point. Then, $(Y,g) \kinsetu (X,f)$.
\end{cor}


\vspace{5mm}
{\sc Acknowledgments.}
The author would like to thank Professor K.~Shiraiwa for the valuable conversations and the suggestions regarding the first version of this manuscript.

\section{Preliminaries}
Although some lemmas in this section are listed in \cite{Shimomura1}, we show the proof here for conveniences.
A compact metrizable totally disconnected perfect space is homeomorphic to $C$.
Therefore, any non-empty closed and open subset of $C$ is homeomorphic to $C$.
Let $\Z$ denote the set of all integers.
Let $V= \{v_1, v_2, \dots, v_n \}$ be a finite set of $n$ symbols with discrete topology. Let $\Sigma(V) := V^{\Z}$ with the product topology.
Then, $\Sigma(V)$ is a compact metrizable totally disconnected perfect space; hence, it is homeomorphic to $C$. We define a homeomorphism $\sigma : \Sigma(V) \to \Sigma(V)$ as
\[ (\sigma(t))(i) = t(i+1)~\text{for all}~i \in \Z,~\text{where}~t = (t(i))_{i \in \Z} \in \Sigma(V). \]
The pair $(\Sigma(V),\sigma)$ is known as a {\it \tfullshiftt} of $n$ symbols.
If a closed set $\Lambda \subseteq \Sigma(V)$ is invariant under $\sigma$, i.e. $\sigma(\Lambda) = \Lambda$, then $(\Lambda,\sigma|_{\Lambda})$ is known as a {\it two-sided subshift}.
In this manuscript, $\sigma|_{\Lambda}$ is abbreviated to $\sigma$.
A finite sequence $u_1 u_2 \cdots u_l$ of elements of $V$ is called a {\it word}\ of length $l$.
For a word $u$ of length $l$ and $m \beint$, we define the cylinder $C_m(u) \subseteq \Lambda$ as
\[ C_m(u) := \{ t \in \Lambda ~|~ t(m + j - 1) = u_{j} \myforall 1 \leq j \leq l \}. \]
%
%
Let $(X,f)$ be a topological dynamical system such that $X$ is homeomorphic to $C$.
Let $\U$ be a finite partition of $X$ by non-empty closed and open subsets.
In this manuscript, we consider partitions that are not trivial, i.e., they consist of more than one element.
We define a directed graph $G = G(f,\U)$ as follows:
\enuma
\item $G$ has the set of vertices $V(f,\U) = \U$
\item $G$ has the set of directed edges $E(f,\U) \subseteq \U \times \U$ such that\[(U,U')\in E(f,\U) \myiff f(U) \cap U' \neq \kuu. \]
\enumz
Note that all elements of $V(f,\U)$ have at least one outdegree and at least one indegree.
Let $G = (V,E)$ be a directed graph, where $V$ is a finite set of vertices and $E \subseteq V \times V$ is a set of directed edges.
$\Sigma(G)$ denotes the \tsubshiftt defined as
\[ \Sigma(G) := \{ t \in V^{\Z} ~|~ (t(i), t(i+1)) \in E \myforall i \in \Z \}. \]
A \tsubshiftt is said to be of {\it finite type} if it is topologically conjugate to $(\Sigma(G),\sigma)$ for some directed graph $G$.
Throughout this manuscript, unless otherwise stated, we assume that all the vertices appear in some element of $\Sigma(G)$, i.e., all the vertices of $G$ have at least one outdegree and at least one indegree.
For the sake of conciseness, we write $(\Sigma(f,\U),\sigma)$ instead of $(\Sigma(G(f,\U)),\sigma)$.
The next lemma follows:
\begin{lem}\label{lem:percon}
Let $(X,f)$ be a topological dynamical system such that $X$ is homeomorphic to $C$.
Let $\U$ be a partition of $X$ by non-empty closed and open subsets of $X$. Then, $\percon{X,f}{\Sigma(f,\U),\sigma}$.
\end{lem}
\begin{proof}
Let $x \in X$ be a periodic point of period $n$ under $f$.
Then, there exists a sequence $\{ U_i\}_{i=0,1,\cdots,n}$ of elements of $\U$ such that $U_n = U_0$ and $f(U_i) \cap U_{i+1} \neq \kuu$ for all $ i = 0,1,\cdots,n-1$.
Thus, $(\Sigma(f,\U),\sigma)$ has a periodic point of period $n$.
\end{proof}
%
\begin{lem}[{\rm Lemma 1.3 of R.~Bowen \cite{Bowen}}]\label{lem:bowen}
Let $G = (V,E)$ be a directed graph.
Suppose that every vertex of $V$ has at least one outdegree and at least one indegree.
Then, $\Sigma(G)$ is topologically mixing if and only if there exists an $N \beposint$ such that for any pair of vertices $u$ and $v$ of $V$, there exists a path from $u$ to $v$ of length $n \geq N$.
\end{lem}
\begin{proof}See Lemma 1.3 of R.~Bowen \cite{Bowen}.
\end{proof}
%
Let $K \subseteq X$.
The diameter of $K$ is defined as $\diam (K) := \sup \{d(x,y) ~|~ x,y \in K\}$.
We define $\mesh(\U) :=  \max \{ \diam(U) ~|~ U \in \U \}$.
%
%
\begin{lem}\label{lem:cont} Let $(X,d)$ be a compact metric space and $f : X \to X$, a continuous mapping.
Then, for any $\epsilon > 0$, there exists $\delta = \delta(f,\epsilon) > 0$ such that
\[\delta < \frac{\epsilon}{2};\] 
\[\text{if}~d(x,y) \leq \delta,~\text{then}~d(f(x),f(y)) < \frac{\epsilon}{2}\ \ \text{for all}\ \ x,y \in X.\]
\end{lem}
\begin{proof}
This lemma directly follows from the uniform continuity of $f$.
\end{proof}
For two directed graphs $G=(V,E)$ and $G'=(V',E')$, $G$ is said to be a {\it subgraph} of $G'$ if $V \subseteq V'$ and $E \subseteq E'$.
\begin{lem}\label{lem:contgraph}
Let $(X,d)$ be a compact metric space; $f : X \to X$, a continuous mapping; and $\epsilon > 0$. Let $\delta = \delta(f,\epsilon)$ be as in lemma \ref{lem:cont} and $\U$, a finite covering of X such that $\mesh(\U) < \delta$.
Let $g~:~X~\to~X$ be a mapping such that $G(g,\U)$ is a subgraph of $G(f,\U)$.
Then, $d(f,g) < \epsilon$.
\end{lem}
\begin{proof}
Let $x \in X$. Then, $x \in U$ and $g(x) \in U'$ for some $U,U' \in \U$. Because $G(g,\U)$ is a subgraph of $G(f,\U)$, there exists a $y \in U$ such that $f(y) \in U'$. Therefore, from lemma \ref{lem:cont}, it follows that
\[d(f(x),g(x)) \leq d(f(x),f(y)) + d(f(y),g(x)) < \frac{\epsilon}{2} + \diam(U') < \epsilon. \]
\end{proof}
From this lemma, the next lemma follows directly.
\begin{lem}\label{lem:subgraph-kinji}
Let $(X,d)$ be a compact metric space; $f : X \to X$, a continuous mapping; and $\{\U_k\}_{k \beposint}$, a sequence of coverings of $X$ such that $\mesh(\U_k) \to 0$ as $k \to \infty$.
Let $\{g_k\}_{k=1,2,\dots}$ be a sequence of mappings from $X$ to $X$ such that $G(g_k,\U_k)$ is a subgraph of $G(f,\U_k)$ for all $k$.
Then, $g_k \to f$ as $k \to \infty$. 
\end{lem}
%
\begin{lem}\label{lem:fine-kinsetu}
Let $(X_1,f_1)$ and $(X_2,f_2)$ be topological dynamical systems such that both $X_1$ and $X_2$ are homeomorphic to $C$.
Let $\{ \U_k \}_{k \beposint}$ be a sequence of finite partitions by non-empty closed and open subsets of $X_1$ such that $\mesh(\U_k) \to 0$ as $k \to \infty$.
Let $\{\pi_k\}_{k \beposint}$ be a sequence of continuous commuting mappings from $X_2$ to $X_1$.
Suppose that for all $k \beposint$, $\pi_k(X_2) \cap U \nekuu$ for all $U \in \U_k$.
Then, $(X_2,f_2) \kinsetu (X_1,f_1)$.
\end{lem}
\begin{proof}
Let $k \beposint$. Let $U \in \U_k$.
Because $\pi_k(X_2) \cap U \nekuu$, $\pi_k^{-1}(U)$ is a non-empty closed and open subset of $X_2$.
Both $\pi_k^{-1}(U)$ and $U$ are homeomorphic to $C$.
Therefore, there exists a homeomorphism $\psi_k : X_2 \to X_1$ such that $\psi_k(\pi_k^{-1}(U)) = U$ for all $U \in \U_k$.
Because $\pi_k$ is commuting, $\pi_k(f_2(\pi_k^{-1}(U))) \cap U' \nekuu $ only if $f_1(U) \cap U' \nekuu $.
Let $g_k = \conjugacy{f_2}{\psi_k}$.
Then, from the construction of $\psi_k$, $G(g_k,\U_k)$ is a subgraph of $G(f_1,\U_k)$. 
Because $k \beposint$ is arbitrary, from lemma \ref{lem:subgraph-kinji}, we get the result.
\end{proof}
%
%
%
\begin{lem}\label{lem:sequence-kinsetu}
Let $(X_1,f_1)$ and $(X_2,f_2)$ be topological dynamical systems.
Let $(Y_k,g_k)$ $(k \beposint)$ be a sequence of topological dynamical systems.
Suppose that there exists a sequence of homeomorphisms $\psi_k : Y_k \to X_1$ such that $\conjugacy{g_k}{\psi_k} \to f_1$ as $k \to \infty$ and that $(X_2,f_2) \kinsetu (Y_k,g_k)$ for all $k = 1,2,\dots$.
Then, $(X_2,f_2) \kinsetu (X_1,f_1)$.
\end{lem}
\begin{proof}
Let $\epsilon > 0$.
There exists an $N \beposint$ such that $d(\conjugacy{g_k}{\psi_k}, f_1) < \epsilon / 2$ for all $k > N$.
Fix $k > N$.
Let $\delta > 0$ be such that if $d(y,y') < \delta$, then $d(\psi_k(y),\psi_k(y')) < \epsilon / 2$.
Because $(X_2,f_2) \kinsetu (Y_k,g_k)$, there exists a homeomorphism $\psi' : X_2 \to Y_k$ such that $d(\conjugacy{f_2}{\psi'},g_k) < \delta$.
Then, we find that $d(\conjugacy{f_2}{(\psi_k \circ \psi')},f) < d( \psi_k \circ (\psi' \circ f_2 \circ \psi'^{-1}) \circ \psi_k^{-1},\conjugacy{g_k}{\psi_k}) + d(\conjugacy{g_k}{\psi_k},f_1) < \epsilon$.
\end{proof}
%
\begin{lem}\label{lem:mix-Cantor}
Let $G = (V,E)$ be a directed graph.
Suppose that every vertex of $G$ has at least one outdegree and at least one indegree.
Suppose that $\Sigma(G)$ is topologically mixing and that $\Sigma(G)$ is not a single point.
Then, $\Sigma(G)$ is homeomorphic to $C$.
\end{lem}
\begin{proof}
Suppose that $\Sigma(G)$ is topologically mixing.
Then, by lemma \ref{lem:bowen}, there exists an $N \beposint$ such that for any pair $u$ and $v$ of vertices of $G$, there exists a path from $u$ to $v$ of length $n$ for all $n \geq N$.
Then, it is easy to check that every point $t \in \Sigma(G)$ is not isolated.
Hence, $\Sigma(G)$ is homeomorphic to $C$.
\end{proof}
%
%
\begin{lem}[{\rm Krieger's Marker Lemma, (2.2) of M.~Boyle} \cite{Boyle}]\label{lem:Krieger}
Let $(\Lambda,\sigma)$ be a \tsubshift. Given $k > N > 1$, there exists a closed and open set $F$ such that
\enuma 
\item the sets $\sigma^l(F), 0 \leq l < N$, are disjoint, and
\item if $t \in \Lambda$ and $t_{-k}\dots t_{k}$ is not a $j$-periodic word for any $j<N$, then
\[ t \in \bigcup_{-N < l < N} \sigma^l(F)\ . \]
\enumz
\end{lem}
\begin{proof}
See M.~Boyle \cite[(2.2)]{Boyle}.
\end{proof}
%
%
The next lemma is essentially a part of the proof of the extension lemma given by M.~Boyle \cite[(2.4) ]{Boyle}.
%
%
Although the outlook of the next lemma seems slightly strengthened from Lemma 3.4 in \cite{Shimomura1}, the proof is almost same.
We show the proof only for completeness.

\begin{lem}\label{lem:marker-mix}
Let $(\Sigma,\sigma)$ be a mixing \tsubshiftt of finite type.
Let $W$ be a finite set of words that appear in some elements of $\Sigma$.
Then, there exists an $M \beposint$ that satisfies the following condition:
\begin{itemize}
\item if $(\Lambda,\sigma)$ is a \tsubshiftt such that $\percon{\Lambda,\sigma}{\Sigma,\sigma}$ and $\Lambda$ has either a non-periodic orbit or a periodic orbit of least period greater than $M$, then there exists a continuous shift-commuting mapping $\pi : \Lambda \to \Sigma$ such that there exists a $t \in \pi(\Lambda)$ in which all words in $W$ appear as segments of $t$.
\end{itemize}
\end{lem}
\begin{proof}
$\Sigma$ is isomorphic to $\Sigma(G)$ for some directed graph $G = (V,E)$.
Therefore, without loss of generality, we assume that $\Sigma = \Sigma(G)$.
Because $(\Sigma(G),\sigma)$ is a mixing subshift of finite type, there exists an $n>0$ such that for every pair of elements $v,v' \in V$ and every $m \geq n$, there exists a word of the form $v\dots v'$ of length $m$.
In addition, there exists an element $\bar{t} \in \Sigma(G)$ such that $\bar{t}$ contains all words of $W$ as segments.
Let $w_0$ be a segment of $\bar{x}$ that contains all words of $W$.
Let $n_0$ be the length of the word $w_0$.
Let $N = 2n + n_0$.
If $v,v' \in V$ and $m \geq N$, then there exists a word of the form $v \dots w_0 \dots v'$ of length $m$.
Let $k > 2N$.
Let $M > N$.
Note that $N$ depends only on $\Sigma(G)$ and $W$.
Therefore, $M$ also depends only on $\Sigma(G)$ and $W$.
Let $\Lambda$ be a \tsubshiftt such that $\percon{\Lambda,\sigma}{\Sigma,\sigma}$ and $\Lambda$ has either a non-periodic orbit or a periodic orbit of least period greater than $M$.
Using Krieger's marker lemma, there exists a closed and open subset $F \subseteq \Lambda$ such that the following conditions hold:
\enuma
\item\label{item:disjoint} the sets $\sigma^l(F), 0 \leq l < N$, are disjoint;
\item\label{item:periodic} if $t \in \Lambda$ and $t \notin \bigcup_{-N < l < N}\sigma^l(F)$, then $t(-k)\dots t(k)$ is a $j$-periodic word for some $j < N$;
\item the number $k$ is large enough to ensure that if $j$ is less than $N$ and a $j$-periodic word of length $2k+1$ occurs in some element of $\Lambda$, then that word defines a $j$-periodic orbit that actually occurs in $\Lambda$.
\enumz
The existence of $k$ follows from the compactness of $\Lambda$.
Let $t \in \Lambda$.
If $\sigma^i(t) \in F$, then we {\it mark} $t$ at position $i$.
There exists a large number $L>0$ such that whether $\sigma^i(t) \in F$ is determined only by the $2L+1$ block $t(i-L)\dots t(i+L)$.
If $t$ is marked at position $i$, then $t$ is unmarked for position $l$ with $i < l < i+N$.
Suppose that $t(i) \dots t(i')$ is a segment of $t$ such that $t$ is marked at $i$ and $i'$ and $t$ is unmarked at $l$ for all $i < l <i'$.
Then, $i' - i \geq N$.
If $t \in \bigcup_{-N < l < N} \sigma^l(F)$, then $t$ is marked at some $i$ where $-N < i < N$.
Suppose that $t(-N+1)\dots t(N-1)$ is an unmarked segment.
Then, $t \notin  \bigcup_{-N < l < N} \sigma^l(F)$, and according to condition (\ref{item:periodic}), $t(-k)\dots t(k)$ is a $j$-periodic word for some $j < N$.
Suppose that $t(i) \dots t(i')$ is an unmarked segment of length at least $2N-1$, i.e., $i'-i \geq 2N-2$.
Then, for each $l$ with $i+N-1 \leq l \leq i'-N+1$, $t(l-k)\dots t(l+k)$ is a $j$-periodic word for some $j < N$.
Therefore, it is easy to check that $t(i+N-1-k) \dots t(i'-N+1+k)$ is a $j$-periodic word for some $j < N$.
In this proof, we call a maximal unmarked segment an {\it interval}. 
Let $t \in \Lambda$.
Let $\dots t(i)$ be a left infinite interval.
Then, it is $j$-periodic for some $j < N$.
Similarly, a right infinite interval $t(i) \dots $ is $j$-periodic for some $j < N$.
If $t$ itself is an interval, then it is a periodic point with period $j < N$.
If an interval is finite, then it has a length of at least $N-1$.
We call intervals of length less than $2N - 1$ as {\it short} intervals.
We call intervals of length greater than or equal to $2N - 1$ as {\it long } intervals.
If $t$ has a long interval $t(i) \dots t(i')$, then $t(i+N-1-k) \dots t(i'-N+1+k)$ is $j$-periodic for some $j < N$.
We have to construct a shift-commuting mapping $\phi : \Lambda \to \Sigma$.
Let $V'$ be the set of symbols of $\Lambda$.
Let $\Phi : V' \to V$ be an arbitrary mapping.
Let $t \in \Lambda$. Suppose that $t$ is marked at $i$.
Then, we let $(\phi(t))(i)$ be $\Phi(t(i))$.
We map periodic points of period $j < N$ to periodic points of $\Sigma$.
Then, we construct a coding of $\phi(t)$ in three parts.
For any $(v,v',l) \in V\times V \times \{N-1,N,N+1,\dots,2N-2\}$, we choose a word $\Psi(v,v',l)$ in $G$ of length $l$ such that the word of the form $v \Psi(v,v',l) v'$ is a path in $G$.

(A) {\it Coding for short interval:}
Let $t(i) \dots t(i')$ be a short interval. Then, $t$ is marked at $i-1$ and $i'+1$.
We have already defined a code for positions $i-1$ and $i'+1$ as $\Phi(t(i-1))$ and $\Phi(t(i'+1))$, respectively.
The coding for $\{i, i+1, i+2, \dots, i' \}$ is defined by the path $\Psi(\Phi(t(i-1)),\Phi(t(i'+1)),i'-i+1)$.

(B) {\it Coding for periodic segment: }
For an infinite or long interval, there exists a corresponding periodic point of $\Lambda$. The periodic points of $\Lambda$ are already mapped to periodic points of $\Sigma$.
Therefore, an infinite or long periodic segment can be mapped to a naturally corresponding periodic segment.

(C) {\it Coding for transition part: }
To consider a transition segment, let $t(i) \dots t(i')$ be a long interval.
Then, $t(i-1)$ has already been mapped to $\Phi(t(i-1))$, and $t(i+N-1)$ is mapped according to periodic points.
Assume that $t(i+N-1)$ is mapped to $v_0$.
The segment $t(i-1) \dots t(i+N-1)$ has length $N+1$.
We map the segment $t(i) \dots t(i+N-2)$ to $\Psi(\Phi(t(i-1)),v_0,N-1)$.
In the same manner, the transition coding of the right-hand side of a long interval is defined.
Similarly, the transition coding of the left or right infinite interval is defined.

It is easy to check that there exists a large number $L'>0$ such that the coding of $(\phi(t))(i)$ is determined only by the block $t(i-L') \dots t(i+L')$.
Therefore, $\phi : \Lambda \to \Sigma$ is continuous.
Because $\Lambda$ has either a $t \in \Lambda$, which is not a periodic point, or a $t' \in \Lambda$, which is a periodic point of least period greater than $M$, there appears a short interval or transition segment in some elements of $\Lambda$.
In the above coding, we can take $\Psi$ such that both short intervals and transition segments are mapped to words that involve $w_0$.
\end{proof}
\section{Proof of the main result}
\begin{lem}\label{lem:sft-kinji}
Let $X$ be homeomorphic to $C$ and $f$, a chain mixing element of $\End(X)$.
Let $\{\W_k\}_{k \beposint}$ be a sequence of non-trivial finite partitions by non-empty closed and open subsets of $X$ such that $\mesh(\W_k) \to 0$ as $k \to \infty$.
Then, there exists a sequence $\{\psi_k\}_{k \beposint}$ of homeomorphisms from $\Sigma(f,\W_k)$ to $X$ such that $\conjugacy{\sigma}{\psi_k} \to f$ as $k \to \infty$.
Furthermore, if $f$ is chain mixing, then all $(\Sigma(f,\W_k),\sigma)$ $(k \beposint)$ are mixing.
\end{lem}
\begin{proof}
Consider a sequence $\{\W_k\}_{k \beposint}$ of non-trivial partitions of $X$ by non-empty closed and open subsets such that $\mesh(\W_k) \to 0$ as $k \to \infty$.
Assume that $k \beposint$.
Let $G_k = G(f,\W_k)$.
Let $\delta > 0$ be such that if $x,x' \in X$ satisfy $d(x,x') < \delta$, then both $x$ and $x'$ are contained in the same element of $\W_k$.
Let $\{x_0,x_1\}$ be a $\delta$ chain.
Let $U,U' \in \W_k$ be such that $x_0 \in U$ and  $x_1 \in U'$.
Then, $f(U) \cap U' \nekuu$.
Therefore, $(U,U')$ is an edge of $G_k$.
Let $U,V \in \W_k$.
Let $x \in U$ and $y \in V$.
Because $f$ is chain mixing, there exists an $N > 0$ such that for every $n \geq N$, there exists a $\delta$ chain from $x$ to $y$ of length $n$.
Therefore, for every $n \geq N$, there exists a path in $G_k$ from $U$ to $V$ of length $n$.
From lemma \ref{lem:bowen}, $(\Sigma(G_k),\sigma)$ is topologically mixing.
By lemma \ref{lem:mix-Cantor}, $\Sigma(G_k)$ is homeomorphic to $C$.
Therefore, there exists a homeomorphism $\psi_k : \Sigma(G_k) \to X$ such that for any vertex $u$ of $G_k$, $\psi_k(C_0(u)) = u$.
Let $g_k = \conjugacy{\sigma}{\psi_k}$.
Then, by construction, we obtain $G(g_k,\U_k) = G(f,\U_k)$.
Because $\mesh(\U_k) \to 0$ as $k \to \infty$, we conclude that $g_k \to f$ as $k \to \infty$ by lemma \ref{lem:subgraph-kinji}.
\end{proof}

Proof of Theorem \ref{thm:main}
\begin{proof}
Let $X$ and $Y$ be homeomorphic to $C$.
First, suppose that $(Y,g) \kinsetu (X,f)$. Then, it is easy to see that $\percon{Y,g}{X,f}$.
Conversely, suppose that $f \in \End(X)$ is chain mixing; $g \in \Hom(Y)$, aperiodic; and that $\percon{Y,g}{X,f}$.
Let $\{\W_i\}_{i \beposint}$ be a sequence of non-trivial finite partitions by non-empty closed and open subsets of $X$ such that $\mesh(\W_i) \to 0 $ as $i \to \infty$.
By lemma \ref{lem:sft-kinji}, there exists a sequence of homeomorphisms $\psi_i : \Sigma(f,\W_i) \to X$ such that $\conjugacy{\sigma}{\psi_i} \to f$ as $i \to \infty$ and that all $(\Sigma(f,\W_i),\sigma)$ $(i \beposint)$ are mixing.
Fix $i \beposint$.
Let $\Sigma = \Sigma(f,\W_i)$.
Let $\{\U_k\}_{k \beposint}$ be a sequence of finite partitions of $\Sigma$ by non-empty closed and open subsets.
Let $\U_k = \{ U_{k,j} : 1 \leq j \leq n_k \}$ for $k \beposint$.
Then, there exists a sequence $u_{k,j}$ $(k \beposint, 1 \leq j \leq n_k)$ of words and a sequence $m(k,j)$ $(1 \leq j \leq n_k)$ of integers such that the following condition is satisfied:
\[ C_{m(k,j)}(u_{k,j}) \subseteq U_{k,j}\ \ (k \beposint, 1 \leq j \leq n_k). \]
Fix $k \beposint$.
Let $W = \{u_{k,j} ~|~ 1 \leq j \leq n_k \}$.
We shall show the following:
\enuma
\item\label{item:mokuhyo} there exists a continuous commuting mapping  $\barphi_k : Y \to \Sigma$ such that $\barphi_k(Y)$ contains an element $t \in \Sigma$ that contains all words of $W$.
\enumz
Then, $\barphi_k(Y) \cap U \nekuu$ for all $U \in \U_k$.
Because $k \beposint$ is arbitrary, we conclude that $(Y,g) \kinsetu \Sigma$ by lemma \ref{lem:fine-kinsetu}.
Then, by lemma \ref{lem:sft-kinji} and lemma \ref{lem:sequence-kinsetu}, we can conclude that $(Y,g) \kinsetu (X,f)$.

Let $M$ be a positive integer that satisfies the condition in lemma \ref{lem:marker-mix}.
Let $\V$ be a partition of $Y$ by non-empty closed and open subsets.
Then, for each $y \in Y$, there exists a unique $t_y \in \Sigma(g,\V)$ such that $g^l(y) \in t_y(l) \in \V$ for all $l \in \Z$.
Therefore, there exists a commuting mapping $\phi_{\V} : Y \to \Sigma(g,\V)$ such that $\phi_{\V}(y) = t_y$ for all $y \in Y$.
Because all elements of $\V$ are open, it is easy to see that $\phi_{\V}$ is continuous.
Let $\Lambda = \phi_{\V}(Y)$.
Then, $\Lambda$ is a \tsubshift.
Because $\Sigma$ is mixing, there exists an $m \beposint$ such that for all integer $n \geq m$, there exists a periodic point $t_n \in \Sigma$ of period $n$.
If $\V$ is sufficiently fine, then the period $n \in \Per(\Sigma(g,\V),\sigma)$, where $n < m$, has a real periodic point of $(Y,g)$ of period $n$.
Therefore, because $\percon{Y,g}{X,f}$, we get $\percon{\Sigma(g,\V),\sigma}{\Sigma,\sigma}$ for all sufficiently fine $\V$.
Let $\bar{M} > \max \{m,M\}$ be an arbitrary positive integer.
Because $g$ is aperiodic, if $\V$ is sufficiently fine, then $\Lambda$ is not a set of periodic points of period less than $\bar{M}$.
Therefore, by lemma \ref{lem:marker-mix}, there exists a continuous commuting mapping $\pi_k : \Lambda \to \Sigma$ such that $\pi_k(\Lambda)$ contains an element that contains all words of $W$.
Finally, let $\barphi_k = \pi_k \circ \phi_{\V}$; this concludes the proof.
\end{proof}

\end{document}